\documentclass[12pt,reqno]{amsart}
\usepackage{mathrsfs,amsfonts,amsmath}
\usepackage{cite}
\numberwithin{equation}{section}

\newtheorem{theorem}{Theorem}[section]

\newtheorem{lemma}{Lemma}[section]

\newtheorem{definition}{Definition}[section]
\newtheorem{remark}{Remark}[section]

\makeatletter \@addtoreset{figure}{section} \makeatother

 \allowdisplaybreaks

\begin{document}
\title[\ Large time  behavior  of  solution  to nonlinear  Dirac  equation]
      {Large time  behavior  of  solution  to nonlinear  Dirac  equation  in $1+1$ dimensions}
 \author{Yongqian Zhang}
\address[Y. Zhang]{\small School of Mathematical Sciences,
                                   Fudan University, Shanghai 200433, P. R. China}
\email{\tt yongqianz@fudan.edu.cn}
\author{Qin Zhao}
\address[Q. Zhao]{\small School of Mathematical Sciences,
Shanghai Jiao Tong University, Shanghai 200240, P. R. China}
\email{\tt zhao@sjtu.edu.cn}

\begin{abstract}
This paper studies the large time behavior of solution for a class of nonlinear massless Dirac equations in $R^{1+1}$.  It is shown that the solution will tend to travelling wave solution when time tends to infinity.
\end{abstract}
\maketitle

\section{Introduction}
We are concerned with the large time behavior of solution for the nonlinear Dirac equations
\begin{equation}\label{eq-dirac}
\left\{ \begin{array}{l} i(u_t+u_x)=N_1(u,v), \\i( v_t-v_x)=N_2(u,v),
\end{array}
\right.
\end{equation}
with initial data
\begin{equation}\label{eq-dirac-initialv}
(u, v)|_{t=0}=(u_0(x), v_0(x)),
\end{equation}
where $(t,x)\in R^2$, $(u,v)\in \mathbf{C}^2$. The nonlinear terms  take the following form
\begin{equation}\label{eq-nonlinearstruc}
N_1=\partial_{\overline{u}}W(u,v), \quad N_2=\partial_{\overline{v}}W(u,v)
\end{equation}
with
\[
 W(u,v)=\alpha |u|^2|v|^2+\beta (\overline{u}v+u\overline{v})^2,
 \]
where $\alpha,\beta\in R^1$ and $\overline{u}, \overline{v}$ are complex conjugate of $u$ and $v$. (\ref{eq-dirac}) is called  Thirring equation for $\alpha=1$ and $\beta=0$, while it is called Gross-Neveu equation for $\alpha=0$ and $\beta=1/4$; see for instance \cite{thirring,gross-neveu,pelinovsky}.

There are many works devoted to study the Cauchy problem  for nonlinear  Dirac equations, see for example  \cite{bournaveas-zouraris,cacciafesta,candy,contreras,deldado,escobedo,huh,huh2,pelinovsky,sasaki,selberg,zhang} and the references therein. The survey of  well-posedness and stability results for nonlinear Dirac equations in one dimension is given in \cite{pelinovsky}. Recently, the global existence of solutions in $L^2$ for Thirring model in $R^{1+1}$ has been established by Candy  in \cite{candy}, while the well-posedness for solutions with low regularity for Gross-Neveu model  in $R^{1+1}$ has been obtained in Huh and Moon \cite{huh-moon}, and in Zhang and Zhao \cite{zhang-zhao}.

Our aim is to establish the large time behaviour of solution to  (\ref{eq-dirac}) and (\ref{eq-dirac-initialv}). Similar to \cite{Beale}, it is shown that the solution will tend to travelling wave solution when time tends to infinity. The main results are stated as follows.

\begin{theorem}\label{thm-asm}
For any solution $(u,v)\in C([0,\infty),H^s(R^1))$ with $s>\frac{1}{2}$, there hold that,
\begin{equation}\label{eq-thm-asm1} \lim\limits_{t\to\infty} \int_{-\infty}^{\infty} |u(x,t)-u_0(x-t)-G_1(x-t)|^2 dx =0,\end{equation}
and
\begin{equation}\label{eq-thm-asm2} \lim\limits_{t\to\infty} \int_{-\infty}^{\infty} |v(x,t)-v_0(x+t)-G_2(x+t)|^2 dx =0,\end{equation}
where
\[
G_1(x-t)=- i \int_{0}^{\infty} N_1\big( u(x-t+\tau, \tau), v(x-t+\tau,\tau)\big)d\tau,
\]
 and
\[
G_2(x+t)=- i \int_{0}^{\infty} N_2\big( u(x+t-\tau, \tau), v(x+t-\tau,\tau)\big)d\tau.
\]
Moreover, we have
\begin{equation}\label{eq-thm-asm3}
\lim\limits_{t\to\infty} \sup_{x\in R^1} \Big|u(x,t)-u_0(x-t)-G_1(x-t) \Big|=0,
\end{equation}
and
\begin{equation}\label{eq-thm-asm4}
\lim\limits_{t\to\infty} \sup_{x\in R^1} \Big|v(x,t)-v_0(x+t)-G_2(x+t) \Big|=0.
\end{equation}
\end{theorem}

\begin{theorem}\label{thm-asmL2}
For any strong solution $(u,v)\in C([0,\infty),L^2(R^1))$, there exists a pair of functions $(g_1,g_2)\in L^2(R^1)$ such that
\begin{equation}\label{eq-asym1}
 \lim\limits_{t\to\infty} \int_{-\infty}^{\infty} |u(x,t)-u_0(x-t)-g_1(x-t)|^2 dx =0,
\end{equation}
and
\begin{equation}\label{eq-asym2}
\lim\limits_{t\to\infty} \int_{-\infty}^{\infty} |v(x,t)-v_0(x+t)-g_2(x+t)|^2 dx =0.
\end{equation}
\end{theorem}

Here the strong solution to (\ref{eq-dirac}) and (\ref{eq-dirac-initialv}) is defined in following sense.

\begin{definition}\label{def-weaksolution}
A pair of functions $(u,v)\in C([0,\infty); L^2(R^1))$ is called a strong solution to $(\mathrm{\ref{eq-dirac}})$ and $(\mathrm{\ref{eq-dirac-initialv}})$ on $R^1\times [0,\infty)$ if there exits a sequence of classical solutions $(u^{(n)}, v^{(n)}) $ to $(\mathrm{\ref{eq-dirac}})$ on $R^1\times [0,\infty)$ such that
\[ \lim\limits_{n\to\infty} \big( ||u^{(n)}-u||_{L^2(R^1\times [0,T])}+||v^{(n)}-v||_{L^2(R^1\times [0,T])} \big)=0, \]
and
\[ \lim\limits_{n\to\infty} \big( ||u^{(n)}(\cdot,0)-u_0||_{L^2(R^1)}+||v^{(n)}(\cdot,0)-v_0||_{L^2(R^1)} \big)=0,\]
for any $T>0$.
\end{definition}
\begin{remark}
The strong solution has been defined for the linear hyperbolic equations, see for instance, \textnormal{\cite{lax,schecter}}. Here we define the strong solution for nonlinear dirac equations in the same way as in \textnormal{\cite{schecter}}.
\end{remark}
\begin{remark}
For any sequence of classical solutions $(u^{(n)}, v^{(n)}) $ given as in Definition \textnormal{\ref{def-weaksolution}}, it has been shown in \textnormal{\cite{zhang-zhao}} that for any $T>0$, there holds the following,
\[ \lim\limits_{n\to\infty} ||u^{(n)}v^{(n)}-uv||_{L^2(R^1\times [0,T])}=0.\]
Therefore, the strong solution $(u,v)$ is also a weak solution to $(\mathrm{\ref{eq-dirac}})$ and $(\mathrm{\ref{eq-dirac-initialv}})$ in distribution sense.
\end{remark}

The remaining of the paper is organized as follows. In Section 2 we first recall Huh's result on the global well-posedness in $C([0,\infty); H^s(R^1))$ for $s>1/2$ for (\ref{eq-dirac}).  Then we  use the characteristic method to establish asymptotic estimates on the global classical solutions and prove Theorem \ref{thm-asm}. In section 3 we prove Theorem \ref{thm-asmL2}  for the case of strong solution.

\section{Asymptotic estimates on the global classical solutions }
The global existence of the solution in $C([0,+\infty); H^s(R^1))\cap C^1([0,+\infty);H^{s-1}(R^1))$ to (\ref{eq-dirac}) and (\ref{eq-dirac-initialv}) for $s>\frac{1}{2}$ has been obtained by Huh in \cite{huh2}. We consider the smooth solution $(u,v)\in C^1 (R^1\times [0,\infty))$ to (\ref{eq-dirac}) and (\ref{eq-dirac-initialv}) in this section.

Multiplying the first equation of (\ref{eq-dirac}) by $\overline{u}$   and the second equation by $\overline{v}$ gives
\begin{equation}\label{eq-dirac1}
\left\{\begin{array}{l}
(|u|^2)_t +(|u|^2)_x=2\Re (i\overline{N_1}u), \\  (|v|^2)_t-(|v|^2)_x=2\Re (i\overline{N_2}v),
\end{array}\right.
\end{equation}
 which, together with the structure of nonlinear terms, leads to the following,
\begin{equation}\label{eq-dirac-conserv}
(|u|^2+|v|^2)_t +(|u|^2-|v|^2)_x=0.
\end{equation}

Now we consider the solution  $(u,v)$ in the triangle domain. For any $a, b\in R^1$ with $a<b$ and for any $t_0\ge 0$, we denote
\[
\Delta(a,b,t_0)=\{ (x,t)\big|\, a-t_0+t<x<b+t_0-t, \, t_0<t<\frac{b-a}{2}+t_0\},
\]
 see Fig. \ref{fig-domain}.
\begin{figure}[h]
\begin{center}
\unitlength=10mm
\begin{picture}(10,3.5)
\thicklines
\put(0,0){\line(1,0){10}}
\put(1,0){\line(5,4){4}}
\put(9,0){\line(-5,4){4}}
\put(0,-0.5){$(a,t_0)$}
\put(9,-0.5){$(b,t_0)$}\put(10.3,0){$t=t_0$}
\put(3.5,3.5){$(\frac{a+b}{2},\frac{b-a}{2}+t_0)$}
\put(4,1.5){$\Delta(a,b,t_0)$}
\end{picture}
\caption{Domain $\Delta(a,b,t_0)$}\label{fig-domain}
\end{center}
\end{figure}
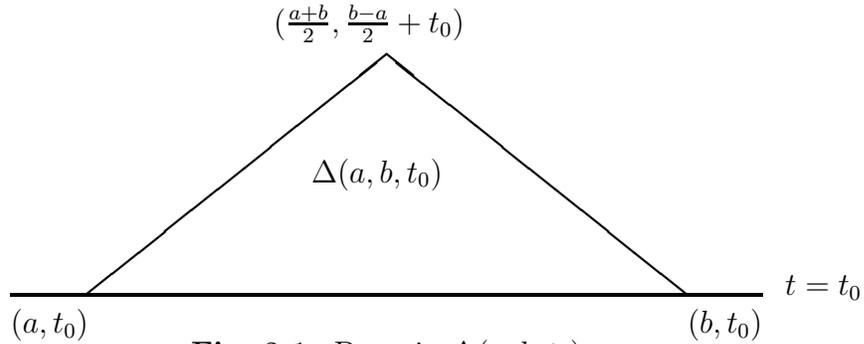
It is obvious that $\Delta(a,b,t_0)$ is bounded by two characteristic lines and $t=t_0$. The vertices of $\Delta(a,b,t_0)$ are $(a,t_0)$, $(b,t_0)$ and $(\frac{a+b}{2}, \frac{b-a}{2}+t_0)$.

\begin{lemma}\label{lemma-conv-charge-1}
For any $\tau\in [t_0, \frac{b-a}{2}+t_0]$, there holds that
\begin{eqnarray*}
\int_{a-t_0+\tau}^{b+t_0-\tau} \big( |u(x,\tau)|^2+|v(x,\tau)|^2\big)dx +2 \int_{t_0}^{\tau} |u(b+t_0-s,s)|^2 ds \\ + 2\int_{t_0}^{\tau} |v(a-t_0+s,s)|^2 ds = \int_a^b \big( |u(x,t_0)|^2 +|v(x,t_0)|^2\big) dx.
\end{eqnarray*}
\end{lemma}
\begin{proof}
As in  Huh \cite{huh2}, we can get the result by taking the integration  of (\ref{eq-dirac-conserv}) over the domain
\[\Omega(a,b,t_0, \tau)=\{ (x,t)\big|\, a-t_0+t<x<b+t_0-t, \, t_0<t<\tau\}. \]
The proof is complete.
\end{proof}
A special example of Lemma \ref{lemma-conv-charge-1} is the following estimate on the domain $\Delta(x_0-t_0, x_0+t_0,0)$ for any $x_0\in R^1$ and $t_0>0$,
\begin{equation}\label{eq-conser-charge-1}
 2\int_0^{t_0} |u(x_0+t_0-s,s)|^2ds
 +2\int_0^{t_0} |v(x_0-t_0+s,s)|^2ds = \int_{x_0-t_0}^{x_0+t_0} (|u_0(x)|^2+|v_0(x)|^2)dx.
\end{equation}

With above lemma, we can derive the following pointwise estimates on the triangle via the characteristic method.

\begin{lemma}\label{lemma-PointwiseEstimate}
Suppose that $\int^{\infty}_{-\infty}(|u_0(x)|^2+|v_0(x)|^2)dx< C_0$ for some constant $C_0>0$. Then there hold that
\begin{equation}\label{eq-pointwise1}
|u(x,t)|^2 \le e^{8|\beta|C_0}\ |u_0(x-t)|^2 ,
\end{equation}
and
\begin{equation}\label{eq-pointwise2}
|v(x,t)|^2\le e^{8|\beta|C_0} |v_0(x+t)|^2,
\end{equation}
for $x\in R^1,\,\, t\ge 0$.
\end{lemma}
 \begin{remark} We remark that the estimates $(\mathrm{\ref{eq-pointwise1}})$ and $(\mathrm{\ref{eq-pointwise2}})$ in Lemma $\mathrm{\ref{lemma-PointwiseEstimate}}$ have been proved by Huh in $\textup{\cite{huh2}}$ for $R^1\times [0,\infty)$ and we give the sketch of the proof here. Indeed, $(\mathrm{\ref{eq-dirac1}})$  gives
\begin{eqnarray*}
\frac{d}{ds} |u(x-t+s, s)|^2 \le 8|\beta| |u(x-t+s,s)|^2|v(x-t+s,s)|^2 ,
\end{eqnarray*}
and
\begin{eqnarray*}
\frac{d}{ds} |v(x+t-s, s)|^2 \le 8|\beta| |u(x+t-s,s)|^2|v(x+t-s,s)|^2.
\end{eqnarray*}
Then, taking the integration of the above from $0$ to $t$ yields that
\begin{eqnarray*}
|u(x,t)|^2 \le \exp(  8|\beta|\int_0^t|v(x-t+s,s)|^2 ds)
|u_0(x-t)|^2 ,
\end{eqnarray*}
which implies the estimate $(\mathrm{\ref{eq-pointwise1}})$ on $u$ by $(\mathrm{\ref{eq-conser-charge-1}})$. The estimate $(\mathrm{\ref{eq-pointwise2}})$ on $v$ could be derived in the same way.
\end{remark}

\begin{lemma}\label{lemma-L2bound}
There exists a constant $C>0$, such that for any $t\ge 0$, there hold the following,
\begin{eqnarray*}
\int^{\infty}_{-\infty} |F_1(y,t)|^2 dy \le C \int^{\infty}_{-\infty} |u_0(y)|^2 \Big(\int^{\infty}_{y+2t} |v_0(\tau)|^2d\tau\Big)^2 dy,
\end{eqnarray*}
and
\begin{eqnarray*}
\int^{\infty}_{-\infty} |F_2(y,t)|^2 dy \le C \int^{\infty}_{-\infty} |v_0(y)|^2 \Big(\int_{-\infty}^{y-2t} |u_0(\tau)|^2d\tau\Big)^2 dy,
\end{eqnarray*}
 where
\begin{align*}
&F_1(y,t)=\int^{\infty}_t |u(y+s,s)||v(y+s,s)|^2 ds,
\end{align*}
and
\begin{align*}
&F_2(y,t)=\int^{\infty}_t |v(y-s,s)||u(y-s,s)|^2 ds.
\end{align*}
\end{lemma}
\begin{proof}
By Lemma \ref{lemma-PointwiseEstimate}, we have
\begin{align*}
|F_1(y,t)|^2 &\le  C |u_0(y)|^2 \Big( \int^{\infty}_t |v_0(y+2s)|^2 ds \Big)^2,
\end{align*}
for some constant $C>0$. Then
\begin{align*}
\int^{\infty}_{-\infty} |F_1(y,t)|^2 dx
&\le  C \int^{\infty}_{-\infty} |u_0(y)|^2 \Big( \int^{\infty}_{t} |v_0(y+2s)|^2 ds \Big)^2 dy \\
&\le  C \int^{\infty}_{-\infty} |u_0(y)|^2 \Big( \int^{\infty}_{y+2t} |v_0(\tau)|^2 d\tau \Big)^2 dy.
\end{align*}
The inequality for $F_2$ could be proved in the same way. The proof is complete.
\end{proof}

\begin{lemma}\label{lemma-asymL2}
There hold that
\[
\lim\limits_{t\to+\infty} \int^{\infty}_{-\infty} \Big(\int^{\infty}_t |u(x-t+s,s)||v(x-t+s,s)|^2 ds \Big)^2 dx=0,
 \]
 and
 \[
\lim\limits_{t\to+\infty} \int^{\infty}_{-\infty} \Big(\int^{\infty}_t |v(x+t-s,s)||u(x+t-s,s)|^2 ds \Big)^2 dx=0.
\]
\end{lemma}
\begin{proof}
By Lemma \ref{lemma-L2bound}, we have
\begin{align*}
&\int^{\infty}_{-\infty} \Big(\int^{\infty}_t |u(x-t+s,s)||v(x-t+s,s)|^2 ds \Big)^2 dx \\
= &\int^{\infty}_{-\infty} \Big(\int^{\infty}_t |u(y+s,s)||v(y+s,s)|^2 ds \Big)^2 dy \\
\le {}&  C \int^{\infty}_{-\infty} |u_0(y)|^2 \Big(\int^{\infty}_{y+2t} |v_0(\tau)|^2d\tau\Big)^2 dy,
\end{align*}
Since
\[
\sup\limits_{y,t} \int_{y+2t}^{\infty} |v_0(\tau)|^2 d\tau\le \int_{-\infty}^{\infty} |v_0(\tau)|^2d\tau,
\]
then the first inequality follows by Lebesgue's dominant convergence Theorem. The second inequality could be proved in the same way. The proof is complete.
\end{proof}

Before proving Theorem \ref{thm-asm}, let us state the following lemma.
\begin{lemma}\label{lemma-initialdata-asymp}
For $(u_0,v_0)\in H^s(R^1)$ with $s>\frac{1}{2}$, then it holds that
\[
\lim\limits_{|x|\to +\infty} (|u_0(x)|+|v_0(x)|)=0.
\]
\end{lemma}
\begin{proof}
In fact, we have
\begin{align*}
\int_{R^1}|\widehat{u}_0(\xi)|d\xi&=\int_{R^1}|\widehat{u}_0(\xi)|(1+|\xi|^2)^{\frac{s}{2}}
(1+|\xi|^2)^{-\frac{s}{2}}d\xi\\
&\leq \left(\int_{R^1}|\widehat{u}_0(\xi)|^2(1+|\xi|^2)^{s}d\xi\right)^{\frac{1}{2}}
\left(\int_{R^1}(1+|\xi|^2)^{-s}d\xi\right)^{\frac{1}{2}},
\end{align*}
where
\[
\widehat{u}_0(\xi)=\int_{R^1} u_0(x)e^{-ix\xi}dx.
\]
Then we get $\widehat{u}_0\in L^1(R^1)$ for $u_0\in H^s(R^1)$  with $s>\frac{1}{2}$.

Therefore, for $x\neq 0$, we can obtain
\begin{align*}
\int_{R^1} \widehat{u}_0(\xi)e^{ix\xi}d\xi&=-\int_{R^1} \widehat{u}_0(\xi)e^{ix(\xi+\frac{\pi}{x})}d\xi=-\int_{R^1} \widehat{u}_0(\xi-\frac{\pi}{x})e^{ix\xi}d\xi,
\end{align*}
and
\begin{align*}
\lim\limits_{|x|\to +\infty} |u_0(x)|&=\lim\limits_{|x|\to +\infty} \left|\int_{R^1} \widehat{u}_0(\xi)e^{ix\xi}d\xi\right|\\
&=\lim\limits_{|x|\to +\infty}\frac{1}{2} \left|\int_{R^1} (\widehat{u}_0(\xi)-\widehat{u}_0(\xi-\frac{\pi}{x}))e^{ix\xi}d\xi\right|\\
&\leq\lim\limits_{|x|\to +\infty}\frac{1}{2}\int_{R^1} \left|\widehat{u}_0(\xi)-\widehat{u}_0(\xi-\frac{\pi}{x})\right|d\xi\\
&=0.
\end{align*}
 The result on $v_0$ could be proved in the same way. The proof is complete.
 \end{proof}

{\bf Proof of Theorem \ref{thm-asm}.} First we use the characteristic method to derive the following,
\begin{equation}\label{eq-thm-asm-p1}
u(x,t)-u_0(x-t)- G_1(x-t)=i\int_t^{\infty}  N_1\big( u(x-t+\tau, \tau), v(x-t+\tau,\tau)\big)d\tau,
\end{equation}
where $G_1(x-t)$ is given by Theorem \ref{thm-asm}, and
\[ \big|N_1\big( u(x-t+\tau, \tau), v(x-t+\tau,\tau)\big) \big|\le C_* |u(x-t+\tau,\tau)||v(x-t+\tau,\tau)|^2,
\]
for some constant $C_*>0.$

Then the asymptotic estimate (\ref{eq-thm-asm1}) for $u$ follows by Lemma \ref{lemma-asymL2}, and the estimate (\ref{eq-thm-asm2}) for $v$ could be derived in the same way.

 Now to prove (\ref{eq-thm-asm3}) and (\ref{eq-thm-asm4}), we need to estimate the remainder term in (\ref{eq-thm-asm-p1}) in $L^{\infty}$. We first use Lemma \ref{lemma-initialdata-asymp} to derive that for any $\varepsilon>0$, there is a constant $M<0$, depending on $\varepsilon$ and a constant $C>0$, such that
\[
\sup\limits_{y\le M} |u_0(y)| \le \varepsilon,
\]
which, together with Lemma \ref{lemma-PointwiseEstimate}, leads to the following,
\begin{align*}
& \sup\limits_{x\le M} \Big|\int_{t}^{\infty}N_1\big( u(x-t+\tau, \tau), v(x-t+\tau,\tau)\big)d\tau \Big|  \\
\le{} & C_* \sup\limits_{x\le M}\int^{\infty}_t |u(x-t+\tau,\tau)||v(x-t+\tau,\tau)|^2 d\tau \\
\le{} &C C_*\sup\limits_{x\le M}|u_0(x-t)|\int_{-\infty}^{\infty}|v_0(x)|^2 dx \\
\le{} &CC_*C_0\sup\limits_{x\le M}|u_0(x-t)| \\
\le{} &CC_*C_0\varepsilon.
\end{align*}
On the other hand, we can obtain
\begin{align*}
& \sup\limits_{x\ge M} \Big|\int_{t}^{\infty}N_1\big( u(x-t+\tau, \tau), v(x-t+\tau,\tau)\big)d\tau \Big| \\
\le{} & C_*\sup\limits_{x\ge M}\int^{\infty}_t |u(x-t+\tau,\tau)||v(x-t+\tau,\tau)|^2 d\tau \\
\le{} & C C_*\sup\limits_{x\ge M}\int^{\infty}_t |u_0(x-t)||v_0(x-t+2\tau)|^2 d\tau \\
\le{} & C C_* ||u_0||_{L^{\infty}}\sup\limits_{x\ge M}\int^{\infty}_{x+t} |v_0(y)|^2 dy \\
\le{} & C C_* ||u_0||_{L^{\infty}}\int^{\infty}_{M+t} |v_0(y)|^2 dy.
\end{align*}

Therefore, with the above estimates, we have
\[ \limsup\limits_{t\to+\infty} \sup\limits_{x\in R^1} \Big|\int_{t}^{\infty}N_1\big( u(x-t+\tau, \tau), v(x-t+\tau,\tau)\big)d\tau \Big| \le \varepsilon,\]
for any $\varepsilon>0$,
which yields
\[
\lim\limits_{t\to+\infty} \sup\limits_{x\in R^1} \Big|\int_{t}^{\infty}N_1\big( u(x-t+\tau, \tau), v(x-t+\tau,\tau)\big)d\tau \Big| =0.\]
Then it holds that
\[
\lim\limits_{t\to\infty} \sup_{x\in R^1} \Big|u(x,t)-u_0(x-t)-G_1(x-t) \Big|=0.
\]
The  estimate (\ref{eq-thm-asm4}) on $v$ could be proved in the same way. The proof is complete.\qed

\section{The case of strong solution}
Due to \cite{zhang-zhao},  for the strong solution $(u,v)\in C([0,\infty), L^2(R^1))$, there is a sequence of smooth solutions, $(u^{(k)}, v^{(k)})$ of (\ref{eq-dirac}), satisfying the following,
 \begin{equation}\label{eq-conv1}  \lim_{k\to \infty}\max_{t\in[0,T]} \Big(||u^{(k)}(\cdot, t)-u(\cdot,t)||_{L^2(R^1)}+ ||v^{(k)}(\cdot, t)-v(\cdot,t)||_{L^2(R^1)}\Big)=0,\end{equation}
\begin{equation}\label{eq-conv2} \lim_{k,l\to \infty}||u^{(k)}v^{(k)}-u^{(l)}v^{(l)}||_{L^2(R^1\times [0,T])} =0,
 \end{equation}
and
\begin{equation}\label{eq-conv3} \lim_{k,l\to \infty}||\overline{u^{(k)}}v^{(k)}-\overline{u^{(l)}}v^{(l)}||_{L^2(R^1\times [0,T])} =0.  \end{equation}

\begin{lemma}\label{lemma-remainder}
For $j=1,2$, there hold that
\[
\lim\limits_{t\to +\infty} \sup\limits_{k\ge 0} \int_{-\infty}^{\infty} \big|F_j^k(y,t)\big|^2dy=0,
\]
where
\[ F_1^k(y,t)=\int_t^{\infty} N_1\big(u^{(k)}(y+\tau,\tau), v^{(k)}(y+\tau,\tau)\big)d\tau, \] and
\[ F_2^k(y,t)=\int_t^{\infty} N_2\big(u^{(k)}(y-\tau,\tau), v^{(k)}(y-\tau,\tau)\big)d\tau. \]
\end{lemma}
\begin{proof}
Let $(u_0^{(k)},v_0^{(k)})=(u^{(k)},v^{(k)})\big|_{t=0}$. By (\ref{eq-conv1}), we have
\begin{equation}\label{eq-k-bdd} \sup\limits_{k\ge 0} \big(||u_0^{(k)}||_{L^2}+||v_0^{(k)}||_{L^2}\big) <\infty.
\end{equation}
Then by Lemma \ref{lemma-L2bound}, we can obtain
\begin{align*}
\int^{\infty}_{-\infty} |F_1^k(y,t)|^2 dy\le {} &CC_* \int^{\infty}_{-\infty} |u_0^{(k)}(y)|^2 \Big(\int^{\infty}_{y+2t} |v_0^{(k)}(\tau)|^2d\tau\Big)^2 dy \\
\le{} & CC_* \int^{\infty}_{-\infty} |u_0(y)|^2 \Big(\int^{\infty}_{y+2t} |v_0(\tau)|^2d\tau\Big)^2 dy \\
&+CC_*\big( ||u_0^{(k)}-u_0||^2_{L^2} ||v_0||^4_{L^2}+ ||u_0^{(k)}||^2_{L^2} ||v_0^{(k)}-v_0||^4_{L^2}\big).
\end{align*}

Together with (\ref{eq-conv1}) and (\ref{eq-k-bdd}), it implies that for any $\varepsilon>0$, there exists a constant $K>0$ such that
 \[
 \sup\limits_{k\ge K} \int_{-\infty}^{\infty} \big|F_1^k(y,t)\big|^2dy\le CC_* \int^{\infty}_{-\infty} |u_0(y)|^2 \Big(\int^{\infty}_{y+2t} |v_0(\tau)|^2d\tau\Big)^2 dy+\frac{1}{2}\varepsilon.
 \]
 Therefore, we have
\begin{align*}
 \limsup\limits_{t\to +\infty} \sup\limits_{k\ge 0} \int_{-\infty}^{\infty} \big|F_1^k(y,t)\big|^2dy
\le{}& CC_*\limsup\limits_{t\to +\infty}\int^{\infty}_{-\infty} |u_0(y)|^2 \Big(\int^{\infty}_{y+2t} |v_0(\tau)|^2d\tau\Big)^2 dy\\
&+ \frac{1}{2}\varepsilon + \lim\limits_{t\to +\infty} \sup\limits_{0\le k<K} \int_{-\infty}^{\infty} \big|F_1^k(y,t)\big|^2dy \\
&\le \varepsilon,
\end{align*}
which yields the convergence result for $F_1^{k}$. The convergence for $F_2^{k}$ could be proved in the same way. The proof is complete.
\end{proof}
Direct computation gives the following.
\begin{lemma}\label{lemma-nonlinearterm}
For any $k$ and $j$, there exist a constant $C_1>0$ such that
\[
|N_1(u^{(k)},v^{(k)})-N_1(u^{(j)},v^{(j)})|\le C_1 \big( |u^{(k)}v^{(k)}-u^{(j)}v^{(j)}||v^{(k)}| +|u^{(j)}v^{(j)}||v^{(k)}-v^{(j)}|\big), \]
and
\[
|N_2(u^{(k)},v^{(k)})-N_2(u^{(j)},v^{(j)})|\le C_1 \big( |u^{(k)}v^{(k)}-u^{(j)}v^{(j)}||u^{(k)}| +|u^{(j)}v^{(j)}||u^{(k)}-u^{(j)}|\big). \]
\end{lemma}
Define
\[
G^k_1(y,\infty)=-i\int_0^{\infty} N_1\big(u^{(k)}(y+\tau,\tau), v^{(k)}(y+\tau,\tau)\big)d\tau,
 \]
 and
\[ G^k_2(y,\infty)=-i\int_0^{\infty} N_2\big(u^{(k)}(y-\tau,\tau), v^{(k)}(y-\tau,\tau)\big)d\tau,\]
then we have the following lemma.
\begin{lemma}\label{lemma-asymL2-3}
 There exists a pair of functions  $(g_1,g_2)\in L^2(R^1)$ such that for any $a,b\in R^1$ with $a<b$, there hold that
 \[ \lim\limits_{k\to +\infty} \int_{a}^{b} \Big|G^k_1(y,\infty) -g_1(y) \Big| dy =0,
\]
and
\[ \lim\limits_{k\to +\infty} \int_{a}^{b} \Big|G^k_2(y,\infty) -g_2(y) \Big| dy =0.
\]
Moreover, the sequences of functions $\{G^{k}_1(y,\infty)\}$ and $\{G^{k}_2(y,\infty)\}$ converge weakly to $g_1(y)$ and $g_2(y)$, respectively, in $L^2(R^1)$.
\end{lemma}
\begin{proof}
Let
\[
G^k_1(y,t)=-i\int_0^{t} N_1\big(u^{(k)}(y+\tau,\tau), v^{(k)}(y+\tau,\tau)\big)d\tau,
 \]
then by Lemma \ref{lemma-nonlinearterm}, for any $t<\infty$, we can obtain
\begin{align*}
& ||G^k_1(\cdot,t)-G^j_1(\cdot,t)||_{L^1(R^1)} \\
\le{}& \int_0^t||N_1(u^{(k)},v^{(k)})-N_1(u^{(j)},v^{(j)})||_{L^1(R^1)}(\cdot, \tau) d\tau \\
\le{}& C_1 ||u^{(k)}v^{(k)}-u^{(j)}v^{(j)}||_{L^2(R^1\times [0,t])} ||v^{(k)}||_{L^2(R^1\times [0,t])} \\
& +||u^{(j)}v^{(j)}||_{L^2(R^1\times [0,t])}||v^{(k)}-v^{(j)}||_{L^2(R^1\times [0,t])}.
\end{align*}
Then, from (\ref{eq-conv1})-(\ref{eq-conv3}), it follows that
\begin{equation}
\lim\limits_{k,j\to\infty}||G^k_1(\cdot,t)-G^j_1(\cdot,t)||_{L^1(R^1)}=0.
\end{equation}

Therefore, for any $a,b\in R^1$ with $a<b$, and for any $t<\infty$,  we have
\begin{align*}
& \limsup\limits_{k,j\to\infty}||G^k_1(\cdot,\infty)-G^j_1(\cdot,\infty)||_{L^1((a,b))} \\
\le{}&\limsup\limits_{k,j\to\infty}||G^k_1(\cdot,t)-G^j_1(\cdot,t)||_{L^1((a,b))}
 +\limsup\limits_{k\to\infty} ||F^k_1(\cdot,t)||_{L^1((a,b))}+\limsup\limits_{j\to\infty}||F^j_1(\cdot,t)||_{L^1((a,b))} \\  \le{}& \limsup\limits_{k\to\infty} ||F^k_1(\cdot,t)||_{L^1((a,b))}+\limsup\limits_{j\to\infty}||F^j_1(\cdot,t)||_{L^1((a,b))}.
\end{align*}
Taking $t \to \infty$ in last inequality, it follows from Lemma \ref{lemma-remainder} that
\[ \lim\limits_{k,j\to\infty}||G^k_1(\cdot,\infty)-G^j_1(\cdot,\infty)||_{L^1((a,b))} =0.\]

On the other hand, by Lemma \ref{lemma-remainder}, we have
\[
 \sup\limits_{k\ge 0} \int_{-\infty}^{\infty} \big|G_1^k(y,\infty)\big|^2dy<+\infty.
\]
Then  we can find a function $g_1\in L^2(R^1)$ such that the sequence of functions $\{G^{k}_1(y,\infty)\}$ converges weakly to $g_1(y)$ in $L^2(R^1)$, and
\[ \lim\limits_{k\to\infty}||G^{k}_1(\cdot,\infty)-g_1||_{L^1((a,b))} =0.\]
The result on $G_2^k$ could be proved in the same way. The proof is complete.
\end{proof}

{\bf Proof of Theorem \ref{thm-asmL2}.}
We use the notations in the proof of Lemma \ref{lemma-remainder} and Lemma \ref{lemma-asymL2-3}.

 For smooth solutions $(u^{(k)},v^{(k)})$, we have
 \[
 u^{(k)}(x,t)-u^{(k)}_0(x-t)- G_1^{(k)}(x-t,\infty)=i\int_t^{\infty}  N_1\big( u^{(k)}(x-t+\tau, \tau), v^{(k)}(x-t+\tau,\tau)\big)d\tau
\]
which yields
\[ \int_{-\infty}^{\infty} |u^{(k)}(x,t)-u_0^{(k)}(x-t)-G^k_1(x-t,\infty)|^2 dx \le \sup\limits_{k\ge 0} \int_{-\infty}^{\infty} \big|F_1^k(y,t)\big|^2dy. \]
Then when $k\rightarrow+\infty$, by (\ref{eq-conv1}) and Lemma \ref{lemma-asymL2-3}, we get the following weak convergence in $L^2$,
\[
u^{(k)}(x,t)-u_0^{(k)}(x-t)-G^k_1(x-t,\infty)\rightharpoonup u(x,t)-u_0(x-t)-g_1(x-t),
\]
 in $L^2(R^1)$.

 Therefore,  by the lower semicontinuity of $L^2-$ norm with respect to weak topology in $L^2$, see for example \cite{lax}, we can obtain
\begin{align}\label{eq-asym-approxestimate1}
&\int_{-\infty}^{\infty} |u(x,t)-u_0(x-t)-g_1(x-t)|^2 dx\nonumber\\
\le & \varliminf\limits_{k\to\infty}\int_{-\infty}^{\infty} |u^{(k)}(x,t)-u_0^{(k)}(x-t)-G^k_1(x-t,\infty)|^2 dx\nonumber\\
\le& \sup\limits_{k\ge 0} \int_{-\infty}^{\infty} \big|F_1^k(y,t)\big|^2dy.
\end{align}
Let $t\rightarrow+\infty$,  by Lemma \ref{lemma-remainder}, we can prove (\ref{eq-asym1}). (\ref{eq-asym2}) could be proved in the same way.  The proof is complete.\qed

\section*{ Acknowledgement}
The authors would like to thank the referees for valuable comments and suggestions.
This work was partially supported by NSFC Project 11421061, by the 111 Project
B08018 and by Natural Science Foundation of Shanghai 15ZR1403900.

\end{document}